\documentclass[12pt]{amsart}
\usepackage{amsmath, amssymb, amsfonts}
\usepackage[all]{xypic}
\usepackage[pdftex]{graphicx}
\usepackage{stmaryrd}

\theoremstyle{plain}
\newtheorem{theorem}{Theorem}[section]
\newtheorem{lemma}[theorem]{Lemma}

\newtheorem{prop}[theorem]{Proposition}

\newtheorem{conj}[theorem]{Conjecture}

\newtheorem{theoremA}{Theorem}
\newtheorem{corollaryA}{Corollary}

\theoremstyle{remark}

\newtheorem{remark}[theorem]{Remark}

\newtheorem*{note*}{Note}
\newtheorem*{remark*}{Remark}
\newtheorem*{example*}{Example}

\theoremstyle{definition}
\newtheorem*{definition*}{Definition}
\newtheorem*{hypothesis*}{Hypothesis}
\newtheorem*{assumptions*}{Assumptions}

\newcommand{\Z}{\mathbb{Z}}
\newcommand{\R}{\mathbb{R}}
\newcommand{\Q}{\mathbb{Q}}
\newcommand{\C}{\mathbb{C}}

\newcommand{\Aut}{\mathrm{Aut}}
\newcommand{\Gal}{\mathrm{Gal}}

\newcommand{\cl}{\mathrm{cl}}

\newcommand{\Hom}{\mathrm{Hom}}

\newcommand{\Res}{\mathrm{Res}}

\newcommand{\Irr}{\mathrm{Irr}}
\newcommand{\Ind}{\mathrm{Ind}}

\newcommand{\Spec}{\mathrm{Spec}}
\newcommand{\ram}{\mathrm{ram}}

\numberwithin{equation}{section}

\newcommand{\Fitt}{\mathrm{Fitt}}

\newcommand{\Infl}{\mathrm{Infl}}

\newcommand{\Ext}{\mathrm{Ext}}

\newcommand{\cok}{\mathrm{cok}}

\title[The strong Stark conjecture for totally odd characters]{The strong Stark conjecture\\ for totally odd characters}

\author{Andreas Nickel}
\address{Universit\"{a}t Duisburg--Essen\\
    Fakult\"{a}t f\"{u}r Mathematik\\
    Thea-Leymann-Str. 9\\
    45127 Essen\\
    Germany}
\email{andreas.nickel@uni-due.de}
\urladdr{https://www.uni-due.de/$\sim$hm0251/english.html}

\date{Version of 10th June 2021}
\begin{document}

\begin{abstract}
We prove the $p$-part of the strong Stark conjecture for every totally
odd character and every odd prime $p$. 

Let $L/K$ be a finite Galois CM-extension
with Galois group $G$, which
has an abelian Sylow $p$-subgroup for an odd prime $p$.
We give an unconditional proof of
the minus $p$-part of the equivariant
Tamagawa number conjecture for the pair $(h^0(\Spec(L)), \Z[G])$
under certain restrictions on the ramification behavior in $L/K$.
\end{abstract}

\maketitle

\section*{Introduction}

Let $K$ be a number field and let $\zeta_K(s)$ be the Dedekind zeta function
attached to $K$. If we denote the leading term of its Taylor expansion at $s=0$
by $\zeta_K^{\ast}(0)$, then the analytic class number formula can be rephrased as
\[
	\zeta_K^{\ast}(0) = - \frac{h_K R_K}{w_K},
\]
where $h_K$, $R_K$ and $w_K$ denote the class number, the regulator and the number
of roots of unity in $K$, respectively. Ignoring the sign, this in turn can be
restated as follows: the ratio 
$R_K / \zeta_K^{\ast}(0)$ is rational and generates the fractional ideal
$(w_K/h_K)$.

Denote the absolute Galois group of $K$ by $G_K$ and let $\chi$ be an 
Artin character of $G_K$, i.e.~a character of $G_K$ with open kernel.
Set $\Q(\chi) := \Q(\chi(g) \mid g \in G_K)$,
which is a finite abelian extension of $\Q$.
Chinburg \cite{MR724009} has formulated the so-called strong Stark conjecture
for $\chi$ as a natural refinement of Tate's form of the original
Stark conjecture \cite{MR782485}. Roughly speaking, the latter conjecture
implies that the ratio of a suitable regulator by
the leading term at $s=0$ of the
Artin $L$-series attached to $\chi$ belongs to
the character field $\Q(\chi)$.
Then the strong Stark
conjecture asserts that the principal ideal generated by this ratio
coincides with a certain `$q$-index' that is defined in purely algebraic terms.
For the trivial character this recovers the analytic class number formula
(up to sign) as described above.

Suppose that Stark's conjecture holds. Then the strong Stark conjecture
naturally decomposes into $p$-parts, where $p$ runs over all 
rational primes. More precisely, we say that the $p$-part of the strong
Stark conjecture holds if each prime in $\Q(\chi)$ above $p$ occurs with the
same multiplicity in the prime ideal factorizations of 
the two fractional ideals involved.

Suppose that $K$ is totally real.
Then $\chi$ is called totally odd if
the fixed field $L_{\chi}$ of the kernel of $\chi$
is a totally complex finite Galois extension of $K$ and complex conjugation
induces a unique automorphism $j$ in the center of $\Gal(L_{\chi}/K)$
such that $\chi(j) = - \chi(1)$.
By a celebrated result of Siegel \cite{MR0285488}, Stark's conjecture holds
for totally odd characters.

The main result of this article is the following.

\begin{theoremA} \label{thm:main-theorem}
	Let $K$ be a totally real number field and let $p$ be an odd prime.
	Let $\chi$ be a totally odd Artin character of $G_K$. Then the
	$p$-part of the strong Stark conjecture holds for $\chi$.
\end{theoremA}

As an immediate consequence, we obtain the `weak versions'
of non-abelian generalizations of Brumer's conjecture and
the Brumer--Stark conjecture due to the author \cite{MR2976321}.
This also recovers and generalizes the `non-abelian
Stickelberger theorem' of Burns and Johnston \cite{MR2771125}.

\begin{corollaryA} \label{cor:Brumer-Stark}
The weak Brumer and the weak Brumer--Stark conjecture of
\cite{MR2976321} hold outside their $2$-primary parts.
\end{corollaryA}

For results on the $2$-primary parts of these conjectures we refer the reader
to work of Nomura \cite{MR3408193, MR3208861}.

Now let $L/K$ be a finite Galois extension of number fields
with Galois group $G$. We regard $h^0(\Spec(L))$ as a motive defined
over $K$ and with coefficients in the semisimple algebra $\Q[G]$.
Let $\mathfrak A$ be a $\Z$-order in $\Q[G]$ that contains
the integral group ring $\Z[G]$.
The equivariant Tamagawa number conjecture (ETNC)
for the pair $(h^0(\Spec(L)), \mathfrak A)$ has been formulated
by Burns and Flach \cite{MR1884523} and asserts that a certain
canonical element $T\Omega(L/K, \mathfrak A)$ of the relative
algebraic $K$-group $K_0(\mathfrak A, \R)$ vanishes.
If $\mathfrak A = \mathfrak M(G)$ is a maximal order, then
Burns and Flach \cite[\S 3]{MR1981031} have shown that the ETNC for the
pair $(h^0(\Spec(L)), \mathfrak M(G))$ is equivalent to the strong Stark
conjecture for all irreducible characters of $G_K$ such that
$L_{\chi}$ is contained in $L$.

If $v$ is a place of $K$, we choose a place $w$ of $L$ above $v$ and
write $G_w$ and $I_w$ for the decomposition group and the inertia
subgroup at $w$, respectively. 
If $L/K$ is a CM-extension we let $j \in G$ denote complex conjugation.
As we will see, one may deduce Theorem \ref{thm:main-theorem} from
(the abelian case of) the
following result by adjusting a reduction argument of 
Ritter and Weiss \cite{MR1423032}.

\begin{theoremA} \label{thm:ETNC-holds}
	Let $p$ be an odd prime.
	Let $L/K$ be a Galois CM-extension with Galois group $G$,
	which has an abelian Sylow $p$-subgroup.
	Suppose that every $p$-adic place $v$ of $K$ is at most tamely ramified in $L$
	or that $j \in G_w$. Then the $p$-minus part of the
	ETNC for the pair $(h^0(\Spec(L)), \Z[G])$ holds.
\end{theoremA}

The main ingredients of the proof are a recent result of 
Dasgupta and Kakde \cite{Dasgupta-Kakde} 
on the strong Brumer--Stark conjecture
and a reformulation of the minus part of the ETNC 
due to the author \cite{MR2805422}. Since the latter result
is rather involved and not required if one is only interested in
the strong Stark conjecture, we also give a more direct proof
of Theorem \ref{thm:main-theorem} that almost only relies on \cite{Dasgupta-Kakde}.

By methods of Iwasawa theory, similar results have already been proven by the 
author under a suitable `$\mu_p=0$' condition (\cite[Theorem 5.1]{MR2545262},
\cite[Theorem 4]{MR2805422}
and for not necessarily abelian extensions in \cite[Theorem 1.3]{MR3552493}). Still assuming
that Iwasawa's $\mu_p$-invariant vanishes, Burns \cite{MR4092926}
presented a strategy for verifying the $p$-minus part of the ETNC
in general. 
This approach relies, in addition, on a conjecture of Gross
\cite{MR656068}.
The decisive advantage of our Theorem \ref{thm:ETNC-holds}
is that it does not depend on any conjectural vanishing of $\mu_p$-invariants
or any further conjectures.

If $L$ is abelian over the rationals, the whole ETNC for the pair
$(h^0(\Spec(L)), \Z[G])$ is known by work of Burns and Greither \cite{MR1992015}
and of Flach \cite{MR2863902}. These results also rely on Iwasawa theory
and use that the relevant Iwasawa $\mu_p$-invariants vanish
by a theorem of Ferrero and Washington \cite{MR528968}.

Typical examples of wildly ramified extensions $L/K$ where Theorem
\ref{thm:ETNC-holds} applies may be constructed as follows.
Let $F/K$ be an arbitrary abelian Galois extension of totally real fields
and suppose that the degree $[F:\Q]$ is odd. Then for each non-trivial
$p$-power root of unity $\zeta$ the field $L := F(\zeta)$ is CM
and abelian over $K$. Moreover, we have $j \in G_w$ (actually $j \in I_w$) for 
all $p$-adic places $v$ of $L$ and each such place is wildly ramified
whenever the order of $\zeta$ is sufficiently large.\\

Finally, we record the following consequences of Theorem \ref{thm:ETNC-holds}.

\begin{corollaryA} \label{cor:consequences}
	Let $p$ be an odd prime.
	Let $L/K$ be a Galois CM-extension with Galois group $G$,
	which has an abelian Sylow $p$-subgroup.
	Suppose that every $p$-adic place $v$ of $K$ is at most tamely ramified in $L$
	or that $j \in G_w$.  Then the $p$-parts
	of the following conjectures hold:
	\begin{enumerate}
		\item
		the (non-abelian) Brumer conjecture \cite[Conjecture 2.1]{MR2976321};
		\item 
		the (non-abelian) Brumer--Stark conjecture 
		\cite[Conjecture 2.6]{MR2976321};
		\item
		the minus part of the central conjecture
		(Conjecture 2.4.1) of Burns \cite{MR1863302};
		\item 
		the minus part of the `lifted root number conjecture'
		of Gruenberg, Ritter and Weiss \cite{MR1687551}.
	\end{enumerate}
	If $L/K$ is at most tamely ramified at all places, then
	the $p$-minus parts of both the central conjecture
	(Conjecture 3.3)
	of Breuning and Burns \cite{MR2371375} and of the ETNC for the pair
	$(h^0(\Spec(L))(1), \Z[G])$ hold.
\end{corollaryA}

\subsection*{Acknowledgements}
The author is indebted to Henri Johnston for many useful comments
and for drawing his attention 
to the slides of a talk by Samit Dasgupta from the International Colloquium
on Arithmetic Geometry at Tata Institute of Fundamental Research in Mumbai.
This was the starting point of this project. I am also very grateful to
Samit Dasgupta and Mahesh Kakde for 
helpful correspondence regarding their preprint \cite{Dasgupta-Kakde}.

The author acknowledges financial support provided by the 
Deutsche Forschungsgemeinschaft (DFG) 
within the Heisenberg programme (project number 437113953).

\subsection*{Notation and conventions}
All rings are assumed to have an identity element and all modules are assumed
to be left modules unless otherwise stated. 
Unadorned tensor products will always denote
tensor products over $\Z$.
For every field $F$ we fix a separable closure $F^c$ of $F$ and write
$G_F := \Gal(F^c/F)$ for its absolute Galois group. 

For a finite group $G$ and a prime $p$ we let $\Irr(G)$ and $\Irr_p(G)$ 
be the set of irreducible complex-valued and $\Q_p^c$-valued characters of $G$, respectively.
If $U$ is a subgroup of $G$ and $\chi$ is a character of $U$, we write
$\Ind_U^G \chi$ for the induced character. If $U$ is normal and $\chi$ is 
a character of the quotient $G/U$, then we denote the inflated character
of $G$ by $\Infl^G_{G/U} \chi$.

If $M$ is a $G$-module, we denote the maximal submodule of $M$ upon which $G$ acts trivially by $M^G$.
Similarly, we let $M_G$ denote the maximal quotient of $M$ with trivial $G$-action.
Then $N_G := \sum_{g \in G} g$ induces a map $M_G \rightarrow M^G$ which we also denote by
$N_G$.

For a ring $R$ and a positive
integer $n$, we write $M_n(R)$ for the ring of $n \times n$ matrices with entries in $R$.

\section{The conjectures}

\subsection{General notation}
Let $L/K$ be a finite Galois extension of number fields with Galois group $G$.
For each place $v$ of $K$ we choose a place $w$ of $L$ above $v$ and denote 
the decomposition group and inertia subgroup of $L/K$ at $w$ by $G_w$ and $I_w$,
respectively. When $G$ is abelian, then both $G_w$ and $I_w$ only depend upon $v$,
and we will occasionally also use the notation $G_v$ and $I_v$ in this case.
We let $\phi_w \in G_w/I_w$ be the Frobenius automorphism and denote the cardinality
of the residue field $K(v)$ at $v$ by $N(v)$. For a finite place $w$ of $L$ we write
$\mathfrak P_w$ for the associated prime ideal in $L$.

For a set $S$ of places of $K$ we let $S(L)$ be the set comprising those places of $L$
that lie above a place in $S$. For each prime $p$ we denote the set of $p$-adic places of $K$ by $S_p$
and the set of archimedean places of $K$ by $S_{\infty}$.
When $S$ is finite, we let $Y_{L,S}$ be the free abelian group generated by the places in $S(L)$
and $X_{L,S}$ be the kernel of the augmentation map $Y_{L,S} \rightarrow \Z$ which maps each place
$w \in S(L)$ to $1$. If in addition $S$ contains all archimedean places, then 
$\mathcal{O}_{L,S}$ denotes the ring of $S(L)$-integers in $L$. 
As usual, we abbreviate $\mathcal{O}_{L,S_{\infty}}$ to $\mathcal{O}_{L}$.

\subsection{Artin $L$-series and Stark's conjecture}
Let $S$ be a finite set of places of $K$ containing $S_{\infty}$. For each complex-valued
character $\chi$ of $G$ we denote the $S$-truncated Artin $L$-series attached to $\chi$
by $L_S(s,\chi)$ and the leading coefficient of its Taylor expansion at $s=0$ by
$L_S^{\ast}(0,\chi)$. We choose a $\C[G]$-module $V_{\chi}$ with character $\chi$ and
denote the contragredient of $\chi$ by $\check \chi$.
The negative of the usual
Dirichlet map induces an isomorphism of $\R[G]$-modules
\begin{equation} \label{eqn:Dirichlet-map}
 \lambda_S: \R \otimes \mathcal{O}_{L,S}^{\times} \stackrel{\simeq}{\longrightarrow}
 \R \otimes X_{L,S}, \quad 1 \otimes \epsilon \mapsto - \sum_{w \in S(L)} \log |\epsilon|_w w.
\end{equation}
By the Noether--Deuring theorem (see \cite[Lemma 8.7.1]{MR2392026} for instance) 
there exist (non-canonical) $\Q[G]$-isomorphisms
\begin{equation} \label{eqn:phi_S}
  \phi_S: \Q \otimes X_{L,S} \stackrel{\simeq}{\longrightarrow} \Q \otimes \mathcal{O}_{L,S}^{\times}.
\end{equation}
Each choice of $\phi_S$ induces a $\C$-linear automorphism
\begin{eqnarray*}
 (\lambda_S \circ \phi_S)^{\chi}: \Hom_{\C[G]}(V_{\check \chi}, \C \otimes X_{L,S}) & \longrightarrow &
 \Hom_{\C[G]}(V_{\check \chi}, \C \otimes X_{L,S})\\
 f & \mapsto & \lambda_S \circ \phi_S \circ f.
\end{eqnarray*}
We denote its determinant by $R_S(\chi, \phi_S)$ and set
\[
 A_S(\chi, \phi_S) := \frac{R_S(\chi, \phi_S)}{L_S^{\ast}(0,\chi)} \in \C^{\times}.
\]

\begin{conj}[Stark] \label{conj:Stark}
 For each $\sigma \in \Aut(\C)$ one has 
 $A_S(\chi, \phi_S)^{\sigma} = A_S(\chi^{\sigma}, \phi_S)$.
\end{conj}

Here we let $\sigma \in \Aut(\C)$ act on the left even though we write exponents on the right.
Note that Conjecture \ref{conj:Stark} in particular predicts that
$A_S(\chi, \phi_S)$ lies in the character field 
$\Q(\chi) := \Q(\chi(g) \mid g \in G)$.

\begin{remark}
It is not hard to see that Conjecture \ref{conj:Stark} does not depend on the 
choices of $S$ and $\phi_S$ (see \cite[Chapitre I, Propostion 7.3 and \S 6.2]{MR782485}). 
The truth of the conjecture is invariant under inflation and induction 
and respects addition of characters \cite[Chapitre I, \S 7.1]{MR782485}.
In particular, if we view $\chi$ as an Artin character of $G_K$,
then the conjecture does not depend on the choice of a 
finite extension $L$ that contains the fixed field
$L_{\chi}$ of $\chi$. Moreover, Conjecture \ref{conj:Stark} holds for all characters of $G$
if and only if it holds for all irreducible characters of $G$.
The same observations apply to the strong Stark conjecture 
(Conjecture \ref{conj:Strong-Stark} below).
\end{remark}

The following deep result is due to Siegel \cite{MR0285488}.

\begin{theorem}[Siegel] \label{thm:Siegel}
 Conjecture \ref{conj:Stark} holds for each totally odd character $\chi$.
\end{theorem}

\begin{remark}
We point out that Conjecture \ref{conj:Stark} is known whenever $L_{\chi}$ is abelian
over the rationals \cite{MR1423032}. Using the analytic class number formula, one
can show that it holds for the trivial character. By Brauer induction and a
theorem of Artin one can deduce the conjecture for all rational-valued characters
\cite[Chapitre II, Corollaire 7.4]{MR782485}.
In fact, even the strong Stark conjecture (Conjecture \ref{conj:Strong-Stark} below)
is known in theses cases.
\end{remark}

\subsection{The $q$-index and the strong Stark conjecture}
Let $\mathcal{O}$ be a Dedekind domain with field of fractions $F$. Each finitely generated
torsion $\mathcal{O}$-module $M$ is of the form $M \simeq \bigoplus_{i=1}^n \mathcal{O} / \mathfrak{a}_i$,
where $n$ is a non-negative integer and each $\mathfrak{a}_i$ is an ideal in $\mathcal{O}$.
We set $\ell_{\mathcal{O}}(M) := \prod_{i=1}^n \mathfrak{a}_i$.
If $f:A \rightarrow B$ is a homomorphism of $\mathcal{O}$-modules such that both the kernel and 
cokernel of $f$ are finitely generated torsion $\mathcal{O}$-modules, then the 
$q$-index of $f$ is the fractional ideal
\[
 q(f) := \ell_{\mathcal{O}}(\cok(f)) \cdot \ell_{\mathcal{O}}(\ker(f))^{-1}.
\]
Now suppose the $\mathcal{O}$ is the ring of integers in a finite Galois extension $F$ of $\Q$
which is sufficiently large such that we have a ring isomorphism
\[
 F[G] \simeq \prod_{\chi \in \Irr(G)} M_{\chi(1)}(F).
\]
Choose an injective $\Z[G]$-homomorphism $\phi_S: X_{L,S} \rightarrow \mathcal{O}_{L,S}^{\times}$.
Note that $\phi_S$ then clearly induces a $\Q[G]$-isomorphism as in \eqref{eqn:phi_S}.
For each $\mathcal{O}[G]$-lattice $M$ we define an $\mathcal{O}$-homomorphism
\[
 \phi_{S,M}: \Hom_{\mathcal{O}}(M, \mathcal{O} \otimes X_{L,S})_G \stackrel{N_G}{\longrightarrow}
    \Hom_{\mathcal{O}}(M, \mathcal{O} \otimes X_{L,S})^G \longrightarrow 
    \Hom_{\mathcal{O}}(M, \mathcal{O} \otimes \mathcal{O}_{L,S}^{\times})^G,
\]
where the second arrow is induced by $\phi_S$. Since $F \otimes_{\mathcal{O}} \phi_{S,M}$ is
an isomorphism, the $q$-index $q(\phi_{S,M})$ is defined. If $M'$ is a second $\mathcal{O}[G]$-lattice,
then $q(\phi_{S,M}) = q(\phi_{S,M'})$ whenever 
$F \otimes_{\mathcal{O}} M \simeq F \otimes_{\mathcal{O}} M'$
as $F[G]$-modules by \cite[Chapitre II, Lemme 6.7]{MR782485}. 
Let $\chi$ be a character of $G$. By enlarging $F$ if necessary, 
we may assume that there is an 
$\mathcal{O}[G]$-lattice $M_{\chi}$ such that $F \otimes_{\mathcal{O}} M_{\chi}$ is an 
$F[G]$-module with character $\chi$. Then the fractional ideal
\[
 q_{\phi_S}(\chi) := q(\phi_{S, M_{\chi}})
\]
is well-defined. Since $q_{\phi_S}(\chi^{\sigma}) = q_{\phi_S}(\chi)^{\sigma}$ for each
$\sigma \in \Gal(F/\Q)$, we may actually view $q_{\phi_S}(\chi)$ as a fractional ideal of
$\Q(\chi)$. 
Chinburg \cite[Conjecture 2.2]{MR724009} suggested the following
refinement of Stark's conjecture \ref{conj:Stark}
whenever the set $S$ is sufficiently large.

\begin{conj}[Strong Stark] \label{conj:Strong-Stark}
	Assume that $S$ is a finite set of places of $K$ that contains
	$S_{\infty}$ and all places that ramify in $L/K$. 
	Assume in addition that the finite places in $S$ 
	generate the class group of $L$. Then
	we have that $A_S(\chi, \phi_S)^{\sigma} = A_S(\chi^{\sigma}, \phi_S)$ 
	for each $\sigma \in \Aut(\C)$ and an equality of fractional ideals 
	$q_{\phi_S}(\check \chi) = (A_S(\chi, \phi_S))$ in $\Q(\chi)$.
\end{conj}

%We point out that this conjecture does not depend on the particular choices of $S$
%and $\phi_S$. Moreover, it behaves well under addition, inflation and induction of
%characters.

Suppose that Stark's conjecture holds for $\chi$.
Let $p$ be a prime. We say that the $p$-part of the strong Stark conjecture holds if each 
prime above $p$ occurs with the same multiplicity in the prime ideal 
factorizations of $q_{\phi_S}(\check \chi)$ and $(A_S(\chi, \phi_S))$.

If $\chi$ is an Artin character of $G_K$, we set $G_{\chi}:= \Gal(L_{\chi}/K)$.
If $\chi$ is totally odd, then $L_{\chi}$ is a CM-field. We denote its
maximal totally real subfield by $L_{\chi}^+$ and set 
$G_{\chi}^+:= \Gal(L_{\chi}^+/K)$.

\begin{prop} \label{prop:reduction-to-abelian}
Let $p$ be a prime.
Then the $p$-part of the strong Stark conjecture holds for all totally odd characters 
if and only if
it holds for all totally odd characters $\chi$ such that $G_{\chi}^+$ is cyclic.
If $p$ is odd, one may assume in addition that the cardinality of $G_{\chi}^+$ is
prime to $p$.
\end{prop}

\begin{proof}
	Assume that Stark's conjecture holds for all characters.
	Then it is well-known (see \cite[Proposition 11]{MR1423032}) 
	that the $p$-part of the strong Stark conjecture holds
	for all characters if and only if it holds for all characters such that
	$G_{\chi}$ is cyclic of order prime to $p$. 
	We may adjust the argument for the non-trivial implication of the proposition as follows 
	(see also the proof of \cite[Corollary 2]{MR2805422}).
	
	We first note that Stark's conjecture holds for totally odd characters 
	by Siegel's theorem (Theorem \ref{thm:Siegel}). Fix a prime $p$.
	Let $\chi$ be an arbitrary totally odd character. If $m$ is a positive
	integer, then (the $p$-parts of) two fractional ideals in $\Q(\chi)$ agree if and only if
	their $m$-th powers agree. Thus it suffices to show the 
	$p$-part of the strong Stark conjecture
	for $m \chi$ for an appropriate integer $m>0$.
	Let $\mathcal{F}$ be the family of all subgroups $U$ of $G_{\chi}$ such that
	$U$ contains $j$ and $U/\langle j \rangle$ is cyclic. Then $G_{\chi}$ is the union
	of all subgroups $U$ in $\mathcal{F}$. 
	Thus by Artin's induction theorem \cite[Theorem 17]{MR0450380}
	there is a positive integer $m$ such that
	\[
		m \chi = \sum_{U \in \mathcal{F}} a_U \Ind_U^{G_{\chi}} \phi_U
	\]
	for appropriate $a_U \in \Z$ and linear characters $\phi_U$ of $U$.
	The corresponding intermediate extensions $L/L^U$ are again CM-extensions
	thanks to the condition $j \in U$. Hence each irreducible character of $U$
	is either totally odd or totally even.
	Since the induction of an even (odd) character is again even (odd), the even
	characters in the above expression must add up to $0$. In other words,
	we may assume that all $\phi_U$ are odd. 
	Since the strong Stark conjecture behaves well under addition and induction of characters,
	we may assume that $G_{\chi}^+$ is cyclic. This proves the first claim
	of the proposition.
	
	Let $\mathcal{I}_p$ be the inertia subgroup of the (abelian) extension
	$\Q(\chi)/\Q$ at $p$. Since each $\sigma \in \mathcal{I}_p$ fixes the primes
	above $p$ and Stark's conjecture holds for $\chi$, it suffices to show
	the $p$-part of Conjecture \ref{conj:Strong-Stark} for the character 
	$\chi' = \sum_{\sigma \in \mathcal{I}_p} \chi^{\sigma}$.
	We may and do therefore assume that $p$ is unramified in $\Q(\chi)$.
	Let $P$ be the Sylow $p$-subgroup of $G_{\chi}$.
	Let $\psi$ be an irreducible constituent of $\Res^{G_{\chi}}_P \chi$.
	Then $\Q(\psi)/\Q$ is totally ramified at $p$ so that $\psi^{\sigma}$ is also
	an irreducible constituent of $\Res^{G_{\chi}}_P \chi$ for each $\sigma \in 
	\Gal(\Q(\psi)/\Q)$. Thus $\Res^{G_{\chi}}_P \chi$ is a finite sum of
	characters of the form $\sum_{\sigma \in \Gal(\Q(\psi)/\Q)} \psi^{\sigma}$.
	As the latter is rational-valued (indeed a permutation character)
	and the strong Stark conjecture
	is known for those characters \cite[Chapitre II, Corollaire 7.4]{MR782485}, 
	we may assume that $\Res^{G_{\chi}}_P \chi$ is trivial whenever $p$ is odd.
	Note that the last step indeed fails for $p = 2$ because the resulting
	character would be even.
\end{proof}

\subsection{The equivariant Tamagawa number conjecture} \label{subsec:ETNC}
Let $L/K$ be a finite Galois extension of number fields with Galois group $G$.
Let $S$ be a finite set of places of $K$ that contains all archimedean places
and all places that ramify in $L$.
In \cite{MR1863302} Burns defines a canonical element $T\Omega(L/K,0)$
in the relative algebraic $K$-group $K_0(\Z[G], \R)$. The definition involves
the refined Euler characteristic of a trivialized complex and the leading term
of an equivariant $S$-truncated $L$-series at $s=0$. 
The occurring (perfect) complex is essentially
Tate's canonical class in $\Ext^2_{\Z[G]}(X_{L,S}, \mathcal{O}_{L,S}^{\times})$
and the trivialization is given by the negative of the Dirichlet map
\eqref{eqn:Dirichlet-map}.
We will not give any further details, as we will mainly work with an explicit
reformulation of the author for the case at hand.

The ETNC has been formulated by Burns and Flach \cite{MR1884523} in vast generality,
but for the pair $(h^0(\Spec(L)), \Z[G])$ 
simply asserts the following \cite[Theorem 2.4.1]{MR1863302}.

\begin{conj}
The element $T\Omega(L/K,0)$ in $K_0(\Z[G], \R)$ vanishes.
\end{conj}

We point out that this assertion is equivalent to the `lifted root number
conjecture' of Gruenberg, Ritter and Weiss \cite{MR1687551} by
\cite[Theorem 2.3.3]{MR1863302}. For our purposes, however, it is more
important to note that $T\Omega(L/K,0)$ belongs to the subgroup $K_0(\Z[G], \Q)$
if and only if Stark's conjecture holds for all irreducible characters of $G$
by \cite[Theorem 2.2.4]{MR1863302}. In this case the ETNC decomposes into conjectures
at each prime $p$ by means of the canonical isomorphism
\[
	K_0(\Z[G], \Q) \simeq \bigoplus_p K_0(\Z_p[G], \Q_p)
\]
Let $T\Omega(L/K,0)_p$ be the image of $T\Omega(L/K,0)$ under this isomorphism.
Then $T\Omega(L/K,0)_p$ is torsion if and only if the $p$-part of the strong Stark
conjecture holds by \cite[Theorem 2.2.4]{MR1863302}
(or rather \cite[Lemmas 2.2.6 and 2.2.7]{MR1863302}).

Now assume that $L/K$ is a CM-extension. Then for each odd prime $p$
there is a canonical isomorphism
\[
	K_0(\Z_p[G], \Q_p) \simeq K_0(\Z_p[G]_+, \Q_p) \oplus K_0(\Z_p[G]_-, \Q_p),
\]
where $\Z_p[G]_{\pm} := \Z_p[G]/(1 \mp j)$. By Theorem \ref{thm:Siegel}
we therefore obtain a well-defined element $T\Omega(L/K,0)_p^-$
in $K_0(\Z_p[G]_-, \Q_p)$. We say that the $p$-minus part of the ETNC holds
if the latter element vanishes. In particular, the $p$-part of the strong Stark conjecture
then holds for all odd characters of $G$.

\begin{lemma} \label{lem:2-implies-1}
Theorem \ref{thm:ETNC-holds} for abelian extensions $L/K$
implies Theorem \ref{thm:main-theorem}.
\end{lemma}

\begin{proof}
This follows from the above considerations and Proposition \ref{prop:reduction-to-abelian}.
\end{proof}

\section{Ray class groups}
\subsection{A natural exact sequence} \label{subsec:exact-sequence}
Let $L/K$ be a finite Galois extension of number fields with Galois group $G$.
If $T$ is a finite set of finite places of $K$, we write $\cl_L^T$
for the ray class group of $L$ for the 
modulus $\mathfrak M_T := \prod_{w \in T(L)} \mathfrak P_w$.
Let $S$ be a second finite set of places of $K$ such that $S \cap T$ is empty. We denote
the subset of $S$ comprising all finite places in $S$ by $S_f$. Then there is a natural map
$Y_{L,S_f} \rightarrow \cl_L^T$ which sends each place $w \in S_f(L)$ to the class
$[\mathfrak P_w] \in \cl_L^T$ of the associated prime ideal. We denote the cokernel of
this map by $\cl_{L,S}^T$. 
If $T$ is empty, we usually remove it from the notation.
We set $E_{L,S} := \mathcal{O}_{L,S}^{\times}$ and define
$E_{L,S}^T$ to be the subgroup of those units in $E_{L,S}$ which are congruent to $1$
modulo each place in $T(L)$. We have an exact sequence of $\Z[G]$-modules
\begin{equation} \label{eqn:ray-class-sequence}
	0 \rightarrow E_{L,S}^T \rightarrow E_{L,S} \rightarrow
	(\mathcal{O}_{L,S} / \mathfrak M_T )^{\times} \stackrel{\nu}{\longrightarrow} 
	\cl_{L,S}^T \rightarrow \cl_{L,S} \rightarrow 0,
\end{equation}
where the map $\nu$ lifts an element 
$\overline x \in (\mathcal{O}_{L,S} / \mathfrak M_T )^{\times}$
to $x \in \mathcal{O}_{L,S}$ and sends it to the class of the principal ideal $(x)$
in $\cl_{L,S}^T$.

In the following we assume that no non-trivial root of unity in $L$ is 
congruent to $1$ modulo each place in $T(L)$. In other words, we assume that
$E_{L,S}^T$ is torsion-free. For instance, this condition holds if $T$ contains two places
of different residue characteristic or one place of sufficiently large norm.

Let $L/K$ be a CM-extension. For any $G$-module $M$, we define submodules
\[
	M^{\pm} := \left\{m \in M \mid jm = \pm m \right\}.
\]
In particular, we have that $\mu_L := E_{L,S_{\infty}}^-$ is the group of roots of
unity in $L$, and that $(E_{L,S_{\infty}}^T)^-$ vanishes. We set
\[
	A_{L,S}^T := (\cl_{L,S}^T)^-.
\]
If $p$ is a prime, we set $M(p) := M \otimes \Z_p$ for any finitely generated $\Z$-module $M$.
If $p$ is odd, then taking $p$-minus-parts is an exact functor so that
\eqref{eqn:ray-class-sequence} induces an exact sequence of $\Z_p[G]_-$-modules
\begin{equation} \label{eqn:ray-ses}
0 \rightarrow \mu_L(p) \rightarrow
(\mathcal{O}_{L} / \mathfrak M_T )^{\times}(p)^- \rightarrow 
	A_{L,S_{\infty}}^T(p) \rightarrow A_{L,S_{\infty}}(p) \rightarrow 0.
\end{equation}

%\subsection{Galois descent}
%We keep the notation of \S \ref{subsec:exact-sequence}.
%\begin{lemma} \label{lem:Galois-descent}
%Let $p$ be an odd prime and let $U$ be a normal subgroup of $G$ of order prime to $p$.
%Then we have natural isomorphisms of $\Z_p[G/U]_-$-modules
%\[
%	A_{L,S}^T(p)_U \simeq A_{L^U,S}^T(p).
%\]
%\end{lemma}
%\begin{proof}
%Set $F := L^U$.
%Let $N:A_{L,S}^{T}(p) \rightarrow A_{F,S}^{T}(p)$ and 
%$i:A_{F,S}^{T}(p) \rightarrow A_{L,S}^{T}(p)$ 
%denote the homomorphisms induced by the norm and inclusion maps on ideals, respectively.
%Then $N \circ i : A_{F,S}^{T}(p) \rightarrow A_{F,S}^{T}(p)$ 
%is multiplication by the cardinality of $U$ and thus is an isomorphism of $p$-groups.
%In particular, $i$ is injective and $N$ is surjective. 
%Let $\Delta_p(U)$ denote the kernel of the augmentation map 
%$\Z_p[U] \twoheadrightarrow \Z_p$ which maps each $u \in U$ to $1$. 
%The composite map $i \circ N: A^{T}_{L, S}(p) \rightarrow A^{T}_{L, S}(p)$
%is given by multiplication by $N_U$.
%Since the orders of $U$ and $A^{T}_{L, S}(p)$ are coprime, $A^{T}_{L,S}(p)$ 
%is cohomologically trivial as a $U$-module and thus
%$\ker(i \circ N) =  \Delta_p(U) A_{L,S}^{T}(p)$. 
%As $i$ is injective, we thus have that  $\ker(N) =  \Delta_p(U) A_{L,S}^{T}(p)$.
%Since $N$ is surjective, we conclude that it induces an isomorphism
%$(A_{L,S}^{T}(p))_{U} \simeq A_{F,S}^{T}(p)$ as desired.
%\end{proof}

\subsection{Stickelberger elements}
Let $L/K$ be an \emph{abelian} extension of number fields. Let $S$ and $T$ be two
disjoint finite sets of places of $K$ with $S_{\infty} \subseteq S$.
For each irreducible character $\chi$ of $G$ we define the $S$-truncated
$T$-modified $L$-series of $\chi$ by
\[
	L_S^T(s,\chi) := L_S(s,\chi) \prod_{v \in T} (1 - \chi(\phi_w)N(v)^{1-s}).
\]
We adopt the convention that $\chi(\phi_w) = 0$ if $\chi$ is ramified at $v$.
The Stickelberger element attached to $S$ and $T$ is the unique element
$\theta_{L/K,S}^T \in \C[G]$ such that for all irreducible characters $\chi$ of $G$
one has
\[
	\chi(\theta_{L/K,S}^T) = L_S^T(0,\check \chi).
\]
It follows from Theorem \ref{thm:Siegel} that $\theta_{L/K,S}^T$ actually belongs
to the rational group ring $\Q[G]$. In fact, we will need the following finer result.

\begin{prop} \label{prop:int-Stickelberger}
Let $p$ be a prime and let $L/K$ be an abelian Galois extension of number fields
with Galois group $G$. 
Let $S$ and $T$ be two disjoint finite sets of places of $K$
such that the following conditions are satisfied:
\begin{enumerate}
\item 
The union of $S$ and $T$ contains all non-$p$-adic ramified places;
\item 
the set $S$ contains all $p$-adic places that ramify wildly;
\item 
all archimedean places lie in $S$;
\item 
the group $E_{L,S}^{T_{0}}$ is torsion-free, where $T_{0}$ 
denotes the set of all unramified places in $T$.
\end{enumerate}
Then we have that $\theta_{L/K,S}^T \in \Z_p[G]$.
\end{prop}

\begin{proof}
This follows from \cite[Theorem 5.2]{MR3552493} (take $H=1$)
and heavily relies on results of Pi.\ Cassou-Nogu\`es \cite{MR524276} and of Deligne
and Ribet \cite{MR579702}. See also
\cite[Lemma 1]{MR2805422} for an important special case.
\end{proof}

%For each subgroup $H$ of $G$ we put $N_H := \sum_{h \in H} h \in \Z[G]$.
Let $x \mapsto x^{\sharp}$ be the (anti-) involution on $\Z[G]$ induced by
$g \mapsto g^{-1}$ for each $g \in G$. 
Let $T_0$ be a finite set of places of $K$ that do not ramify in $L$ and are
such that $E_{L,S}^{T_0}$ is torsion-free
(recall that this condition does not depend on $S$). 
Following \cite{MR2371374} and \cite[\S 1]{Dasgupta-Kakde} we define the 
\emph{Sinnott--Kurihara ideal} to be the fractional $\Z[G]$-ideal
\begin{equation} \label{eqn:SK-ideal}
	\mathrm{SKu}^{T_0}(L/K) := (\theta_{L/K,S_{\infty}}^{T_0})^{\sharp}
	\prod_{v \in S_{\ram}} (N_{I_v}, 1 - \phi_w |I_v|^{-1}N_{I_v})
\end{equation}
and note that this is actually contained in $\Z[G]$ 
(see \cite[Lemma 3.4]{Dasgupta-Kakde}).

\subsection{Fitting ideals and the strong Brumer--Stark conjecture}
If $M$ is a finitely presented module over a commutative ring $R$,
we denote the (initial) Fitting ideal of $M$ over $R$ by $\Fitt_R(M)$.
For basic properties of Fitting ideals including the following well-known
lemma we refer the reader to Northcott's excellent book \cite{MR0460383}.

\begin{lemma} \label{lem:Fitting-general}
Let $R$ be a commutative ring and let $M$ and $M'$ be finitely presented
$R$-modules. Then the following hold.
\begin{enumerate}
\item 
We have that $\Fitt_R(M \oplus M') = \Fitt_R(M) \cdot \Fitt_R(M')$.
\item 
If $M \rightarrow M' \rightarrow 0$ is an epimorphism,
then we have an inclusion
\[
	\Fitt_R(M) \subseteq \Fitt_R(M').
\]
\end{enumerate}
\end{lemma}

Let $p$ be  prime and let $G$ be a finite abelian group.
If $M$ is a $\Z_p[G]$-module, we denote the Pontryagin dual
$\Hom_{\Z_p}(M, \Q_p/ \Z_p)$ endowed with the contragredient $G$-action by $M^{\vee}$.
The following lemma is certainly well-known.

\begin{lemma} \label{lem:Fitting-ideals}
Let $p$ be an odd prime and let $L/K$ be a finite abelian CM-extension
with Galois group $G$.
Let $M$ be a finite cohomologically trival $\Z_p[G]_-$-module.
Then the following hold.
\begin{enumerate}
\item 
The Fitting ideal $\Fitt_{\Z_p[G]_-}(M)$ is principal.
\item 
We have an equality
\[
	\Fitt_{\Z_p[G]_-}(M^{\vee}) = \Fitt_{\Z_p[G]_-}(M)^{\sharp}
\]
\item 
The index of the Fitting ideal $\Fitt_{\Z_p[G]_-}(M)$ in $\Z_p[G]_-$
is finite and 
\[
	[\Z_p[G]_- : \Fitt_{\Z_p[G]_-}(M)] = |M|.
\]
\end{enumerate}
\end{lemma}

\begin{proof}
A finite $\Z_p[G]_-$-module is cohmologically trivial if and only if its projective
dimension is at most $1$. Since $\Z_p[G]_-$ is semilocal, any first syzygy of $M$ is free.
This proves (i). For (ii) see \cite[Lemma 6]{MR2046598} 
(note that we have to add the $^{\sharp}$ on the right because of 
our different conventions on the action of $G$ on Pontryagin duals). 
Also (iii) is a standard result (see the proof of
\cite[Theorem 4.11]{MR1750935}, for instance).
\end{proof}

The following deep result has recently been shown by Dasgupta and Kakde 
\cite[Theorem 3.5]{Dasgupta-Kakde}. In particular, it
settles the so-called strong Brumer--Stark conjecture away from $2$.

\begin{theorem} \label{thm:Dasgupta-Kakde}
Let $L/K$ be an abelian CM-extension of number fields with Galois group $G$.
Let $T_0$ be a finite set of places of $K$ that do not ramify in $L$ and are such that
$E_S^{T_0}$ is torsion-free. 
Let $p$ be an odd prime.
Then we have that
\[
 	\Fitt_{\Z_p[G]_-} (A_{L, S_{\infty}}^{T_0}(p)^{\vee}) = \mathrm{SKu}^{T_0}(L/K)(p).
\]
\end{theorem}

\subsection{The strong Stark conjecture for totally odd characters} \label{subsec:strongStark}
For the proof of Theorem \ref{thm:main-theorem} it suffices to verify
Theorem \ref{thm:ETNC-holds} for abelian extensions (see Lemma 
\ref{lem:2-implies-1}). However, the proof of the latter result will build
upon a reformulation of the ETNC due to the author \cite{MR2805422} which is
rather involved. Since Theorem \ref{thm:main-theorem} can be obtained by 
more elementary means, we include a proof here for convenience.
The reader who is mainly interested in Theorem \ref{thm:ETNC-holds}
or already familiar with the results in \cite{MR2805422} 
may jump directly to the next subsection.

\begin{proof}[Proof of Theorem \ref{thm:main-theorem}]
	Let $K$ be a totally real number field and let $p$ be an odd prime.
	By Proposition \ref{prop:reduction-to-abelian} it suffices to
	prove the $p$-part of the strong Stark conjecture for
	totally odd characters of $G_K$ that factor through an
	abelian CM-extension $L/K$ of degree prime to $p$.
	We set $G := \Gal(L/K)$ and	observe that in this case
	every $\Z_p[G]$-module $M$ is cohomologically trivial.
	In particular, the map $N_G: M_G \rightarrow M^G$ is an isomorphism
	and taking $G$-(co)-invariants is an exact functor.
	
	Let $\chi$ be a totally odd irreducible complex character of $G$.
	Each choice of isomorphism $\iota: \C \simeq \C_p$ restricts to an embedding
	$\Q(\chi) \hookrightarrow \Q_p(\chi)$, which corresponds to a choice
	of prime in $\Q(\chi)$ above $p$. 
	We henceforth fix such an isomorphism.
	This allows us to view complex characters as $\C_p$-valued and vice versa.
	The covariant functor
	\[
		M \mapsto M^{\chi} := \Hom_{\Z_p(\chi)}(\Z_p(\chi), \Z_p(\chi) \otimes M)^G
	\]
	is exact on finitely generated $\Z[G]$-modules. 
	Moreover, if $f: A \rightarrow B$ is a homomorphism of
	finitely generated $\Z[G]$-modules with finite kernel and cokernel, then
	one has an equality
	\[
		q(f^{\chi}) = \Fitt_{\Z_p(\chi)}(\cok(f)^{\chi}) \cdot
			\Fitt_{\Z_p(\chi)}(\ker(f)^{\chi})^{-1}.
	\]
	In particular, for every finite set $S$ containing the archimedean places and
	every injective $\Z[G]$-homomorphism
	$\phi_S: X_{L,S} \rightarrow \mathcal{O}_{L,S}^{\times}$ we have that
	\begin{equation} \label{eqn:q-index-Fitt}
		\iota(q_{\phi_S}(\chi)) = \Fitt_{\Z_p(\chi)}(\cok(\phi_S)^{\chi}).
	\end{equation}
	Now assume that $S$ also contains all ramified primes and is large enough
	to generate the class group of $L$.
	We choose an arbitrary $\phi_{S_{\infty}}$ as above. 
	As $\cl_L$ is finite and $\Q[G]$ is semisimple,
	there is always an integer $N$ and a (necessarily injective)
	$\Z[G]$-homomorphism $\phi_S$
	such that the following diagram commutes
	\begin{equation} \label{eqn:phi-diagram}
	 \xymatrix{
	0 \ar[r] & X_{S_{\infty}} \ar[r] \ar[d]^{\phi_{S_{\infty}}} & X_S \ar[r] \ar[d]^{\phi_S}
			& \Z[S_f(L)] \ar[r] \ar[d]^{\cdot N} & 0 \ar[d] & \\
	0 \ar[r] & E_{S_{\infty}} \ar[r] & E_S \ar[r] & \Z[S_f(L)] \ar[r] & \cl_L \ar[r] & 0.
	} 
	\end{equation}
	Since $\chi$ is odd and $X_{S_{\infty}}^-$ vanishes, we obtain the following
	exact sequence of $\Z_p(\chi)$-modules
	\[
		0 \rightarrow \mu_L^{\chi} \rightarrow \cok(\phi_S)^{\chi} \rightarrow
		(\Z/N\Z[S_f(L)])^{\chi} \rightarrow \cl_L^{\chi} \rightarrow 0.
	\]
	Now we choose a second finite set $T_0$ of places of $K$ such that
	$S \cap T_0 = \emptyset$ and $E_{S_{\infty}}^{T_0}$ is torsion-free.
	We deduce from \eqref{eqn:ray-class-sequence} an exact sequence
	\[
		0 \rightarrow \mu_L^{\chi} \rightarrow 
		\left((\mathcal{O}_{L} / \mathfrak{M}_{T_0})^{\times}\right)^{\chi} \rightarrow
		(\cl_L^{T_0})^{\chi} \rightarrow \cl_L^{\chi} \rightarrow 0.
	\]
	The last two exact sequences and \eqref{eqn:q-index-Fitt} imply that
	\begin{equation} \label{eqn:q-index-formula}
		\iota(q_{\phi_S}(\chi)) = \Fitt_{\Z_p(\chi)}((\Z/N\Z[S_f(L)])^{\chi}) \cdot
		\Fitt_{\Z_p(\chi)}\left(\left((\mathcal{O}_{L} / \mathfrak{M}_{T_0})^{\times}\right)^{\chi}\right) \cdot \Fitt_{\Z_p(\chi)}((\cl_L^{T_0})^{\chi})^{-1}.
	\end{equation}
	We next compute the Fitting ideals on the right-hand side. 
	The first two are easily determined. We have
	\begin{eqnarray*}
		\Fitt_{\Z_p(\chi)}((\Z/N\Z[S_f(L)])^{\chi}) & = & ( N^{d_S(\chi)} ), \\
		\Fitt_{\Z_p(\chi)}\left(\left((\mathcal{O}_{L} / \mathfrak{M}_{T_0})^{\times}\right)^{\chi}\right) & = & 
		\left(\prod_{v \in T_0} (1 - \chi(\phi_w) N(v))\right),
	\end{eqnarray*}
	where $d_S(\chi)$ denotes the number of places $v$ in $S_f$ such that
	the restriction of $\chi$ to $G_v$ is trivial. The analytic class number 
	formula implies that (see \cite[Lemma 2.1]{Dasgupta-Kakde})
	\begin{equation} \label{eqn:CN-trick}
		|A_L^{T_0}(p)| = |\Z_p[G]_- / ( \theta_{L/K,S_{\infty}}^{T_0} ) |.
	\end{equation}
	Let $v$ be an arbitrary finite place of $K$.
	Since $\chi(\phi_w)$ is either zero or 
	a root of unity of order coprime to $p$,
	we see that $1 - \chi(\phi_w) \in \Z_p(\chi)^{\times}$ unless
	$G_v$ acts trivially on $V_{\chi}$. In that case we
	have $\chi(N_{I_v}) = |I_v| \in \Z_p^{\times}$
	 so that
	$(\theta_{L/K,S_{\infty}}^{T_0})^{\sharp}$ lies in 
	$\mathrm{SKu}^{T_0}(L/K)(p)$.
	Therefore \eqref{eqn:CN-trick}, Lemma \ref{lem:Fitting-ideals} and Theorem \ref{thm:Dasgupta-Kakde}
	imply that the Fitting ideal $\Fitt_{\Z_p[G]_-}(A_L^{T_0}(p))$
	is principal and generated by $\theta_{L/K,S_{\infty}}^{T_0}$.
	Since the cardinality of $G$ is prime to $p$, the ring $\Z_p[G]_-$
	decomposes into a product $\prod_{\chi(j) = -1} \Z_p(\chi)$.
	So the former statement is equivalent to
	\[
		\Fitt_{\Z_p(\chi)}((\cl_L^{T_0})^{\chi}) = 
		( L_{S_{\infty}}^{T_0}(0, \check \chi) )
	\]
	for all odd $\chi$. So we deduce from \eqref{eqn:q-index-formula}
	and the above that
	\begin{equation} \label{eqn:q-index-formula2}
		\iota(q_{\phi_S}(\chi)) = (N^{d_S(\chi)} / L_{S_{\infty}}(0,\check \chi)).
	\end{equation}
	Finally, we have the following commutative diagram
	whose rows are those of \eqref{eqn:phi-diagram} tensored with $\R$.
	\[ \xymatrix{
		0 \ar[r] & \R \otimes E_{S_{\infty}} \ar[r] \ar[d]^{\lambda_{S_{\infty}}} & 
		\R \otimes E_S \ar[r] \ar[d]^{\lambda_S}
		& \R[S_f(L)] \ar[r] \ar[d]^{\mathrm{Log}_S} & 0  \\
		0 \ar[r] & \R \otimes X_{S_{\infty}} \ar[r] &\R \otimes X_S \ar[r] & \R[S_f(L)] \ar[r] & 0.
	} 
	\]
	Here the first two vertical arrows are the negative Dirichlet maps 
	\eqref{eqn:Dirichlet-map} and
	$\mathrm{Log}_S$ maps $w \in S_f(L)$ to $-\log(N(w))$. For odd $\chi$ one has
	$R_{S_{\infty}}(\chi, \phi_{S_{\infty}}) = 1$ and hence
	\[
		R_S(\chi, \phi_S) = N^{d_S(\chi)}
		\cdot \prod_{v \in S_f \atop V_{\chi}^{G_v} = V_{\chi}} (-\log(N(w))).
	\]
	By \cite[Proposition 6, p.~50]{MR1386895} we likewise have that
	\[
		L_S^{\ast}(0,\chi) = L_{S_{\infty}}(0, \chi)
		\cdot \prod_{v \in S_f \atop V_{\chi}^{G_v} = V_{\chi}} \log(N(w))
		\cdot \prod_{v \in S_f \atop V_{\chi}^{G_v} = 0}
		(1 - \chi(\phi_w)).
	\]
	Since $d_S(\chi) = d_S({\check \chi})$ and $\iota(1 - \chi(\phi_w)) \in
	\Z_p(\chi)^{\times}$ whenever $V_{\chi}^{G_v}$ vanishes, 
	the last two displayed equalities and
	\eqref{eqn:q-index-formula2} complete the proof.
\end{proof}

\subsection{The relation to the equivariant Tamagawa number conjecture}
In order to relate the results of Dasgupta and Kakde 
to the $p$-minus-part of the ETNC, we use the following
reformulation due to the author. Roughly speaking,
its proof is an elaborate refinement of the argument given in
\S \ref{subsec:strongStark}.

\begin{theorem} \label{thm:ETNC-reformulation}
Let $L/K$ be an abelian CM-extension of number fields with Galois group $G$
and let $p$ be an odd prime. 
Assume that each $v \in S_p$ is at most tamely ramified or 
that we have $j \in G_v$.
Let $S_1$ be the set of places of $K$ 
comprising all archimedean places and all wildly ramified $p$-adic places.
Choose an unramified place $v_0$ of $K$ such that $E_{L,S_1}^{T_0}$ is torsion-free,
where we set $T_0 := \{v_0\}$. Moreover, we let $T$ be the union of $T_0$ and
the set of all non-$p$-adic ramified places of $K$. Then the following holds.
\begin{enumerate}
\item 
The attached Stickelberger element is $p$-integral, i.e.\,$\theta_{L/K,S_1}^T \in \Z_p[G]$;
\item 
the $G$-module $A_{L,S_{\infty}}^T(p)$ is cohomologically trivial;
\item
the following are equivalent:
\begin{enumerate}
\item 
the $p$-minus-part of the ETNC for the pair $(h^0(\Spec(L)), \Z[G])$ holds;
\item 
the Fitting ideal $\Fitt_{\Z_p[G]_-}(A_{L,S_{\infty}}^T(p))$ is generated by
$\theta_{L/K,S_1}^T$;
\item 
we have that $(\theta_{L/K,S_1}^T)^{\sharp} \in
\Fitt_{\Z_p[G]_-}(A_{L,S_{\infty}}^T(p)^{\vee})$.
\end{enumerate}
\end{enumerate}
\end{theorem}

\begin{proof}
Part (i) is a special case of Proposition \ref{prop:int-Stickelberger}.
Part (ii) is \cite[Theorem 1]{MR2805422} and the equivalence of (a) and (b)
in part (iii) is a reformulation in terms of
Fitting ideals of \cite[Theorem 2]{MR2805422}. Note that (ii) implies that
the Fitting ideal of $A_{L,S_{\infty}}^T(p)$ is principal 
by Lemma \ref{lem:Fitting-ideals} (i).
Part (ii) of the same lemma shows that (b) implies (c). 
Now suppose that (c) holds. Then the inclusion
$(\theta_{L/K,S_1}^T) \subseteq \Fitt_{\Z_p[G]_-}(A_{L,S_{\infty}}^T(p))$ must 
be an equality, since we have that
\[
	|\Z_p[G]_- : (\theta_{L/K,S_1}^T)| = |A_{L,S_{\infty}}^T(p)|
	= |\Z_p[G]_- : \Fitt_{\Z_p[G]_-}(A_{L,S_{\infty}}^T(p))|.
\]
Here the first equality is \cite[Proposition 4]{MR2805422}
(a variant of \eqref{eqn:CN-trick} and likewise a consequence
of the analytic class number formula),
whereas the second follows from Lemma \ref{lem:Fitting-ideals} (iii).
\end{proof}

\section{Proof of the main results}
We have provided all the ingredients that we need for the proof
of our main results.
We first claim that it suffices to prove Theorem \ref{thm:ETNC-holds}
for abelian extensions.

By Lemma \ref{lem:2-implies-1}, this also finishes 
the proof of Theorem \ref{thm:main-theorem}. 
Moreover, Theorem \ref{thm:main-theorem}
implies Corollary \ref{cor:Brumer-Stark} by
\cite[Theorem 4.1, Proposition 3.9 and Lemma 2.12]{MR2976321}.

By the discussion in \S \ref{subsec:ETNC}
we then also know that $T\Omega(L/K,0)_p^-$ is torsion
for each Galois CM-extension $L/K$.
Hence it follows from \cite[Proposition 6.2]{MR3552493}
that $T\Omega(L/K,0)_p^-$ vanishes if and only if
$T\Omega(L'/K',0)_p^-$ vanishes for all intermediate Galois CM-extensions
$L'/K'$ whose Galois group is either $p$-elementary or a direct product
of a $p$-elementary group and a cyclic group of order $2$ (generated by $j$).
The full strength of Theorem \ref{thm:ETNC-holds} follows 
since the hypotheses on the extension $L/K$ are
inherited by each intermediate Galois CM-extension $L'/K'$
and a $p$-elementary group with abelian Sylow $p$-subgroup is itself abelian.
The main point here is that each subquotient of a finite group $G$ with abelian
Sylow $p$-subgroup $P$ also has an abelian Sylow $p$-subgroup.
To see this, let $U$ be a subgroup of $G$ and let $N$ be a normal subgroup of $U$.
Then by \cite[Aufgabe 26, p.~36]{MR0224703} there is a $g \in G$ such that
$Q := gPg^{-1} \cap U$ is a Sylow $p$-subgroup of $U$. Then $Q$ is abelian
and its image in the quotient $U/N$ is a Sylow $p$-subgroup of $U/N$.

Moreover, Theorem \ref{thm:ETNC-holds} and \cite[Theorem 5.3]{MR2976321}
directly imply the result on the non-abelian
Brumer and Brumer--Stark conjecture listed in Corollary
\ref{cor:consequences}. We have already observed that
the vanishing of $T\Omega(L/K,0)$, the `lifted root number
conjecture' of Gruenberg, Ritter and Weiss and the ETNC for the pair
$(h^0(\Spec(L)), \Z[G])$ all are equivalent.
Hence Theorem \ref{thm:ETNC-holds} implies 
(iii) and (iv) of Corollary \ref{cor:consequences}.
The final claims of Corollary \ref{cor:consequences} on
the conjecture of Breuning and Burns and on the ETNC for the pair
$(h^0(\Spec(L))(1), \Z[G])$ follow from a result of
Bley and Burns \cite[Corollary 6.3(i)]{MR2005875} as in the proof
of \cite[Corollary 1.6]{MR3552493}.

We are left with the following.

\begin{proof}[Proof of Theorem \ref{thm:ETNC-holds} for abelian extensions]
Let $L/K$ be an abelian CM-extension with Galois group $G$ and fix an odd prime $p$.
Suppose that for each $p$-adic place $v$ of $K$ we have that
the cardinality of $I_v$ is prime to $p$ or that $j \in G_v$.
Let $S_1$ be the set of places of $K$ 
comprising all archimedean places and all wildly ramified $p$-adic places.
Choose an unramified place $v_0 \not\in S_p$ such that $E_{L,S_1}^{T_0}$ is torsion-free,
where we set $T_0 := \{v_0\}$. Moreover, we let $T$ be the union of $T_0$ and
the set of all non-$p$-adic ramified places of $K$. By Theorem \ref{thm:ETNC-reformulation}
we have to show that $(\theta_{L/K,S_1}^T)^{\sharp}$ belongs to the Fitting ideal
of $A_{L,S_{\infty}}^T(p)^{\vee}$.

By Theorem \ref{thm:Dasgupta-Kakde} it suffices to check that
$(\theta_{L/K,S_1}^T)^{\sharp}$ lies in the Sinnott--Kurihara ideal
$\mathrm{SKu}^{T_0}(L/K)(p)$. We write
\[
	(\theta_{L/K,S_1}^T)^{\sharp} = (\theta_{L/K,S_{\infty}}^{T_0})^{\sharp} \cdot
	\prod_{v \in S_{\ram} \setminus S_p} (1 - \phi_w N(v) |I_v|^{-1} N_{I_v}) \cdot
	\prod_{v \in S_{\ram} \cap S_p} (1 - \phi_w |I_v|^{-1} N_{I_v}).
\]
If we compare this to \eqref{eqn:SK-ideal}, we see that we are left with showing
\begin{equation} \label{eqn:factor-at-v}
	1 - \phi_w N(v) |I_v|^{-1} N_{I_v} \in (N_{I_v}, 1 - \phi_w |I_v|^{-1} N_{I_v})
\end{equation}
for each place $v \nmid p$ that ramifies in $L/K$.
Let $\ell \not= p$ be the prime below $v$.
By local class field theory \cite[Chapter XV, \S 2]{MR554237}
the local units at $v$ surject under the reciprocity map onto $I_v$. The subgroup
of principal units is mapped onto the Sylow $\ell$-subgroup of $I_v$.
As the factor group of the local units modulo the principal units has
cardinality $N(v)-1$, the ramification index $|I_v|$ divides $N(v)-1$ up to
a $p$-adic unit. Hence \eqref{eqn:factor-at-v} follows from the equality
\[
	1 - \phi_w N(v) |I_v|^{-1} N_{I_v} = 
	1 - \phi_w |I_v|^{-1} N_{I_v} - \phi_w \frac{N(v)-1}{|I_v|} N_{I_v}.
\]

If we jump deeper into the results of Dasgupta and Kakde, we can
alternatively argue as follows.
Let $\mathrm{Sel}_{S_1}^T(L)$ be the Selmer module considered by Burns,
Kurihara and Sano in \cite{MR3522250}. We will not recall its definition
as we only need the following two facts. First there is an epimorphism
of $\Z_p[G]_-$-modules
\[
	\mathrm{Sel}_{S_1}^T(L)(p)^- \rightarrow A_{L, S_{\infty}}^T(p)^{\vee}
	\rightarrow 0
\]
by \cite[Lemma 3.1]{Dasgupta-Kakde}. Second we have that
\[
	\Fitt_{\Z_p[G]_-}(\mathrm{Sel}_{S_1}^T(L)(p)^-) = ((\theta_{L/K,S_1}^T)^{\sharp})
\]
by \cite[Theorem 3.3]{Dasgupta-Kakde}. Hence Lemma \ref{lem:Fitting-general}
implies that we have
\[
	(\theta_{L/K,S_1}^T)^{\sharp} \in \Fitt_{\Z_p[G]_-}(A_{L, S_{\infty}}^T(p)^{\vee}).
\]
as desired.
\end{proof}

\begin{remark}
One may ask whether it is possible to deduce
the $p$-minus part of the ETNC for further non-abelian 
CM-extensions from Theorem \ref{thm:ETNC-holds}
by considering `hybrid' cases as in \cite{MR3461042}.
We claim that this is not the case.
To see this, assume that the $p$-adic group ring $\Z_p[G]$
is `$N$-hybrid' for a normal subgroup $N$ of $G$ in the sense of 
\cite[Definition 2.5]{MR3461042}. Then $p$ does not divide 
the cardinality of $N$ by \cite[Proposition 2.8(i)]{MR3461042}.
Hence each Sylow $p$-subgroup of $G$ is mapped isomorphically
onto a Sylow $p$-subgroup of $G/N$ under the natural quotient map
$G \rightarrow G/N$. Thus if $G/N$ has an abelian Sylow $p$-subgroup
so does $G$.
\end{remark}

%\nocite*
\bibliography{StrongStark-Bib}{}
\bibliographystyle{amsalpha}

\end{document}